\documentclass[sn-mathphys,Numbered]{sn-jnl}

\usepackage{graphicx}
\usepackage{amsmath}
\usepackage{amssymb}
\usepackage{amsfonts}
\usepackage{mathrsfs}
\usepackage{mathtools}
\usepackage{setspace}

\theoremstyle{thmstyleone}
\newtheorem{theorem}{Theorem}[section]

\theoremstyle{thmstylethree}
\newtheorem{definition}{Definition}[section]

\theoremstyle{thmstyletwo}
\newtheorem{remark}{Remark}[section]

\makeatletter
\@addtoreset{equation}{section}
\makeatother
\renewcommand{\theequation}{\thesection.\arabic{equation}}
\newcommand{\subclass}[1]{\par\addvspace\smallskipamount\noindent\textbf{Mathematics Subject Classification (2020):}\enspace#1}

\begin{document}

\title[Analytical Evaluation of Ramanujan Series for $1/\pi$]{Analytical Evaluation of Ramanujan Series for $1/\pi$ via Degree-2, Degree-3, Degree-7, and Degree-19 Modular Transformations}

\author{\fnm{Pablo} \sur{Fern\'andez Refolio}}
\email{pablezfr@gmail.com}

\affil{\orgdiv{Independent Researcher}, \orgaddress{\city{Madrid}, \country{Spain}}}

\abstract{We provide an explicit analytical evaluation of Ramanujan-type series for $1/\pi$. Focusing on the singular moduli $k_{r}$ for $r \in \{5, 7, 13, 37\}$, we demonstrate that the underlying elliptic identities can be established through lower-degree modular transformations. Specifically, we resolve the case $r = 5$ via a combination of degree-2 and degree-3 modular transformations; the case $r = 7$ utilizing the modular transformation of degree 2; the case $r = 13$ via a combination of degree-2 and degree-7 modular transformations; and the case $r = 37$ via a combination of degree-2 and degree-19 modular transformations.}


\maketitle

\subclass{33C05, 33E05, 11F03} 

\section{Introduction}
The evaluation of rapidly converging series for $1/\pi$ constitutes one of the most brilliant chapters in Srinivasa Ramanujan's 1914 paper \cite{ramanujan1914}. At the heart of his derivations lies the theory of elliptic functions and modular equations. This rich mathematical framework traces its origins back to Leonhard Euler's discovery of the addition theorem for elliptic integrals, which was subsequently augmented by John Landen's celebrated transformation \cite{landen1775}. The systematic classification of these integrals into canonical forms was later achieved by Adrien-Marie Legendre \cite{legendre1825}. However, the field underwent a profound paradigm shift through Carl Friedrich Gauss's independent discovery of the arithmetic-geometric mean (AGM) \cite{almkvistberndt1988,borwein1987}, alongside Carl Gustav Jacob Jacobi's \cite{jacobi1829} revolutionary inversion of these integrals into fully fledged elliptic functions. Crucially, Thomas Clausen's 1828 hypergeometric product identity \cite{clausen1828} provided the essential analytical bridge, allowing the squaring of these complete elliptic integrals and mapping ${}_2F_1$ functions directly into the ${}_3F_2$ hypergeometric series that underpin modern $\pi$-evaluations.

Historically, the earliest breakthrough in this family of formulas belongs to Gustav Bauer \cite{bauer1859}, who in 1859 proved a remarkable alternating series for $2/\pi$ via Fourier--Legendre expansions. Decades later, and seemingly unaware of Bauer's prior proof, Ramanujan famously included this exact identity as a central showcase of his discoveries in his legendary first letter to G. H. Hardy in 1913.

While Ramanujan left a spectacular list of seventeen rapidly converging series for $1/\pi$, he rarely provided explicit proofs for the numerical coefficients involved. The monumental task of reconstructing and editing these proofs was heavily advanced by Bruce Berndt in his multi-volume analysis of Ramanujan's notebooks \cite{berndt1985, berndt1989, berndt1991, berndt1994, berndt1998} and his lost notebook \cite{andrewsberndt2005}, uncovering the deep analytical structure behind the formulas. Furthermore, the foundational work of Jonathan and Peter Borwein \cite{borwein1987} systematically formalized the connection between these series and the Arithmetic-Geometric Mean (AGM). Crucial to their framework is the introduction of the elliptic $\alpha$-function, defined at the singular value $k_{r}$ as
\begin{equation}
\alpha(r) = \frac{\pi}{4 K^2(k_r)} + \sqrt{r} - \frac{E(k_r) \sqrt{r}}{K(k_r)},
\end{equation}
where $K(k)$ and $E(k)$ denote the complete elliptic integrals of the first and second kind, respectively. This singular invariant elegantly encapsulates the algebraic structure required to construct rapidly converging Ramanujan-type formulas. Building directly upon the foundational framework and the elliptic $\alpha$-function parameterization established by the Borwein brothers \cite{borwein1987}, David and Gregory Chudnovsky subsequently generalized these principles to develop their celebrated, ultra-fast-converging algorithms for the high-precision computation of $\pi$ \cite{chudnovsky1989ramanujan}.

In recent decades, this paradigm was further advanced through the pioneering work of the late Jesús Guillera \cite{guillera2002WZ} via the introduction of Wilf--Zeilberger (WZ) pairs \cite{wilfzeilberger1990}. Utilizing higher hypergeometric functions, Guillera discovered entirely new classes of Ramanujan-like series, thereby opening new vistas in the modern theory of hypergeometric summation.

In this work, we challenge this approach by demonstrating that a structured decomposition utilizing lower-degree modular transformations yields complete, transparent, and hand-calculable derivations. By constructing a unified approach leveraging modular transformations of degrees 2, 3, 7, and 19, we present a transparent verification for the cases $r=5, 7, 13$, and $37$. This provides an alternative strategy for deriving modular identities to obtain certain Ramanujan series via the introduction of complex numbers.

\section{Preliminaries and Results}
The Euler Gamma function, denoted by $\Gamma(z)$, is defined via the following convergent improper integral for $\Re(z) > 0$
\begin{equation*}\label{eq:gamma_integral}
    \Gamma(z) = \int_0^\infty t^{z-1} e^{-t} \, dt.
\end{equation*}

To frame the complete elliptic integrals within the theory of hypergeometric functions, we recall Gauss's celebrated hypergeometric series~${}_2F_1(a,b;c;x)$ \cite{gauss1813}, which is defined for $0 < |x| < 1$ as
\begin{equation*}\label{eq:gauss_hypergeometric}
    {}_2F_1(a,b;c;x) = \sum_{m=0}^{\infty} \frac{(a)_m (b)_m}{(c)_m} \frac{x^m}{m!},
\end{equation*}
where $(a)_m = \Gamma(a+m)/\Gamma(a)$ denotes the rising factorial or Pochhammer symbol.
The complete elliptic integrals of the first and second kind, denoted by $K(k)$ and $E(k)$ respectively, admit the following equivalent representations via power series and hypergeometric expansions
\begin{equation*}\label{eq:K_expansion}
    K(k)  \coloneqq  \int_{0}^{\pi/2} \frac{dt}{\sqrt{1-k^2\sin^2t}} = \frac{\pi}{2}\sum_{m=0}^{\infty}\frac{(2m)!^2}{2^{4m}m!^4}k^{2m} = \frac{\pi}{2} \, {_2F_1}\left(\frac{1}{2}, \frac{1}{2}; 1; k^2\right),
\end{equation*}
with $k \in (0, 1)$ and
\begin{equation*}\label{eq:E_expansion}
    E(k)  \coloneqq  \int_{0}^{\pi/2} \sqrt{1-k^2\sin^2t} \, dt = \frac{\pi}{2}\sum_{m=0}^{\infty}\frac{(2m)!^2}{2^{4m}(1-2m)m!^4}k^{2m} = \frac{\pi}{2} \, {_2F_1}\left(-\frac{1}{2}, \frac{1}{2}; 1; k^2\right), 
\end{equation*}
with $k \in (0, 1]$.
These integrals fundamentally satisfy Legendre \cite{legendre1825} relation
\begin{equation}\label{eq:legendre_relation}
    K(k)E(k') + K(k')E(k) - K(k)K(k') = \frac{\pi}{2},
\end{equation}
where $k' = \sqrt{1-k^2}$ represents the complementary modulus.
The standard derivatives of $K(k)$ and $E(k)$ are given by
\begin{equation}\label{eq:derke}
  \frac{\mathrm{d}K}{\mathrm{d}k}=\frac{E(k)}{k(1-k^2)} - \frac{K(k)}{k},\hspace{.5cm}\frac{\mathrm{d}E}{\mathrm{d}k}=\frac{E(k)-K(k)}{k}.
\end{equation}
Furthermore, these complete elliptic integrals satisfy Landen descending transformations \cite{almkvistberndt1988, borwein1987, landen1775}
\begin{equation}\label{eq:landendownwardke}
    K(\frac{2\sqrt{k}}{1+k})=K(k)(1+k),\hspace{.5cm}E(\frac{2\sqrt{k}}{1+k})=\frac{2E(k)}{1+k}+K(k)(k-1),
\end{equation}
the Landen ascending transformations
\begin{equation}\label{eq:landeupwardke}
   K\left(\frac{1-k'}{1+k'}\right)= \frac{1+k'}{2}K(k),\hspace{.5cm}E(\frac{1-k'}{1+k'})=\frac{E(k)}{1+k'}+\frac{k'K(k)}{1+k'},
\end{equation}
and the Jacobi imaginary transformations
\begin{equation}\label{eq:jacobi_transke}
    K(ik) = \frac{1}{\sqrt{1+k^2}}K\left(\frac{k}{\sqrt{1+k^2}}\right),\hspace{.5cm}E(ik)=\sqrt{1+k^2}E\left(\frac{k}{\sqrt{1+k^2}}\right).
\end{equation}
When restricted to the real domain, the equality is satisfied unconditionally for all $k \in \mathbb{R}$.
Replacing  $k \to \frac{k}{k'}$ in \eqref{eq:jacobi_transke} we have
\begin{equation}\label{eq:jacobi_transke2}
K\left(\frac{ik}{k'}\right) = k'K(k),\hspace{.5cm}E\left(\frac{ik}{k'}\right) = \frac{1}{k'} E(k),
\end{equation}
with $k \in (0, 1)$. In this domain, the complementary modulus $k'$ remains real-valued and bounded such that $0 < k' < 1$, ensuring that the formulas map consistently to real outputs.
The complete elliptic integral of the first kind satisfies Jacobi's reciprocal modulus transformation. For a complex modulus $k$ in the right half-plane ($\Re(k) > 0$) avoiding the branch cuts, the identity is conventionally defined on its principal branch as
\begin{equation}
K\left(\frac{1}{k}\right) = k \left( K(k) - i K(k') \right). \label{eq:jacobireciprocal}
\end{equation}
\begin{definition}[Degree-$n$ Modular Relation]\label{eq:modular_relation_degree}
Let $k$ and $l$ be two complex elliptic moduli. Suppose that their corresponding complete elliptic integrals of the first kind satisfy the relation
\begin{equation}
n \frac{K(k')}{K(k)} = \frac{K(l')}{K(l)},
\end{equation}
where $k' = \sqrt{1-k^2}$ and $l' = \sqrt{1-l^2}$ denote the complementary moduli defined via their principal branches. Then, $l$ is said to be of degree $n$ over $k$.
\end{definition}
\begin{definition}
The Dedekind \cite{dedekind1877} eta function $\eta(\tau)$ is a complex-valued function defined on the upper half-plane $\mathbb{H} = \{\tau \in \mathbb{C} : \Im(\tau) > 0\}$ by the following infinite product
\begin{equation}\label{eq:dedekind_eta}
\eta(q) \coloneqq q^{1/24} \prod_{m=1}^{\infty} (1 - q^m),
\end{equation}
where $q = e^{2\pi i \tau}$ is the modular parameter. Furthermore, let $\tau = i \frac{K(k')}{K(k)}$ where $K(k)$ is the complete elliptic integral of the first kind. The square of the Dedekind eta function, $\eta^2(q)$, can be expressed explicitly in terms of $K(k)$ and the elliptic modulus $k$ via the relation
\begin{equation}\label{eq:eta_squared_elliptic}
\eta^2(q) = \frac{2K(k)}{\pi} \cdot 2^{-2/3} k^{1/3} (1 - k^2)^{1/6}.
\end{equation}
\end{definition}
Historically, this remarkable formula is a consequence of the pioneering work of Carl Gustav Jacob Jacobi \cite{jacobi1829}. Building upon this framework, we shall introduce the Ramanujan--Weber modular invariants to systematically analyze the arithmetic properties of the underlying singular moduli.
\begin{definition}
Let $q = e^{-\pi \sqrt{n}}$ where $n$ is a positive rational number. The Ramanujan--Weber modular invariants \cite{weber1895,ramanujan1914}, denoted as $G_n$ and $g_n$, are defined in terms of the Dedekind eta function $\eta(\tau)$ as
\begin{equation}\label{eq:Gn_definition}
G_n \coloneqq 2^{-1/4} q^{-1/24} \prod_{m=0}^{\infty} (1 + q^{2m+1}) = 2^{-1/4} \frac{\eta\left(\frac{1+i\sqrt{n}}{2}\right)}{\eta(i\sqrt{n})}.
\end{equation}
\begin{equation}\label{eq:gn_definition}
g_n \coloneqq 2^{-1/4} q^{-1/24} \prod_{m=0}^{\infty} (1 - q^{2m+1}) = 2^{-1/4} \frac{\eta\left(\frac{i\sqrt{n}}{2}\right)}{\eta(i\sqrt{n})}.
\end{equation}
\end{definition}
\begin{remark}
These invariants are algebraic numbers when $n$ is an integer. They are connected to the complete elliptic integral of the first kind $K(k)$ via the modulus $k$, where $g_n$ and $G_n$ can be expressed in terms of the singular modulus $k_n$ as
\begin{equation}
G_n = (2k_n k_n')^{-1/12}, \quad g_n = \left(\frac{2k_n}{k_n'^2}\right)^{-1/12},
\end{equation}
where $k_n' = \sqrt{1 - k_n^2}$ represents the complementary modulus.
\end{remark}
Furthermore, by inverting the classical definition of Ramanujan's class invariant the elliptic modulus $k_{n}$ can be recovered explicitly in terms of $G_n$ via the relation
\begin{equation}\label{eq:ramanujaninvariantrelation}
k_{n} = \frac{1}{\sqrt{2}} \sqrt{1 - \sqrt{1 - G_{n}^{-12}}}.
\end{equation}
Alternatively, using Ramanujan's second class invariant, the elliptic modulus $k$ can be expressed in an remarkably elegant form via the relation
\begin{equation*}
k_{n} = g_{n}^{6} \sqrt{g_{n}^{12} + g_{n}^{-12}} - g_{n}^{12}.
\end{equation*}
Let $q = e^{-\pi K(k')/K(k)}$ be the modular nome associated with the elliptic modulus $k$, where $K(k)$ denotes the complete elliptic integral of the first kind. The derivative of the nome with respect to the modulus is given by
\begin{equation}\label{eq:dqdk}
    \frac{dq}{dk} = \frac{\pi^2 q}{2 k (1-k^2) K^2(k)}.
\end{equation}
Equation~\eqref{eq:dqdk} emerges as a direct consequence of applying the chain rule and Legendre's relation~\eqref{eq:legendre_relation}, a fundamental identity originally due to Jacobi~\cite{jacobi1829}.
\begin{definition}
The Eisenstein--Ramanujan series $P(q)$, which corresponds to the normalized weight-2 Eisenstein \cite{eisenstein1847} series $E_2(\tau)$ under the change of variable $q = e^{2\pi i \tau}$, is defined on the punctured unit disk $0 < |q| < 1$ by the following infinite series
\begin{equation}
    P(q):=24 q \frac{d}{dq} \log{\eta(q)} = 1 - 24 \sum_{m=1}^{\infty} \frac{m q^m}{1 - q^m}.
\end{equation}
\end{definition}
We introduce the following classic formulas from the theory of elliptic functions, which have been extensively investigated by numerous authors and are treated in depth by Berndt et al.(see, e.g., \cite{baruahberndt2009eisenstein, baruahberndtChan2009}). These expressions establish the modular transformation properties and the algebraic relationships between the series $P(q)$ and the complete elliptic integrals
\begin{equation}\label{eq:reflection}
    \tau P(e^{-2\pi\sqrt{\tau}})+P(e^{-2\pi/\sqrt{\tau}})=\frac{6\sqrt{\tau}}{\pi},\hspace{.5cm}\Re(\sqrt{\tau}) > 0.
\end{equation}
Furthermore, by considering the modular transformations of degree 2, where the operators act on the duplicated nome $q^2$ i.e  Landen ascending transformation. Utilizing \eqref{eq:landeupwardke},\eqref{eq:jacobi_transke2},\eqref{eq:eta_squared_elliptic}, performing logarithmic  differentation, and using \eqref{eq:dqdk} we have the following algebraic relations involving the complete elliptic integrals $K(k)$ and $E(k)$
\begin{equation}\label{eq:degree2_nome1}
    2P(q^{2})-P(-q)=\frac{4K^{2}(k)}{\pi^2}(1-2k^2),
\end{equation}
\begin{equation}\label{eq:degree2_nome2}
    P(q^{2})=\frac{12E(k)K(k)}{\pi^2}+\frac{(4k^2-8)K^{2}(k)}{\pi^2},
\end{equation}
and
\begin{equation}\label{eq:degree2_nome3}
2P(q^{2})-P(q)=\frac{4K^2(k)}{\pi^2}(1+k^2).
\end{equation}
The sign inversion $q \to -q$ is done by the modular transformation mapping $k \to  \frac{ik}{k'}$. We shall employ the following modular identities, as recorded by Ramanujan in \cite{ramanujan1914} (Table III).
The modular identity of degree 3
\begin{equation}\label{eq:deg3}
3P(q^{6}) - P(q^{2}) = \frac{4K(k)K(l)}{\pi^2}(1 + kl + k'l'),
\end{equation}
the modular identity of degree 7
\begin{equation}\label{eq:deg7}
7P(q^{14}) - P(q^{2}) = \frac{12K(k)K(l)}{\pi^2}(1 + kl + k'l'),
\end{equation}
as well as his remarkable counterpart of degree 19
\begin{equation}\label{eq:deg19}
19P(q^{38}) - P(q^{2}) = \frac{24K(k)K(l)}{\pi^2}\left(1 + kl + k'l' + \sqrt{kl} + \sqrt{k'l'} - \sqrt{kk'll'}\right),
\end{equation}
where $k' = \sqrt{1-k^2}$ and $l' = \sqrt{1-l^2}$ denote the complementary moduli. These modular transformations are discussed in full detail by Berndt \cite[Chs.~19 and 20]{berndt1991}.
As a consequence of Clausen's hypergeometric product identity, which maps a $_{2}F_{1}$ function to a $_{3}F_{2}$ series, alongside the transformation theory connecting signatures 2 and 4, it follows that for $k \in (0, 3-2\sqrt{2}]$ we have
\begin{equation}\label{eq:clausen4}
K^2(k)= \frac{\pi^2}{4(1-2k^2)} \,_3F_2\left(\frac{1}{4}, \frac{1}{2}, \frac{3}{4}; 1, 1; \frac{16k^2(k^2-1)}{(2k^2-1)^4}\right).
\end{equation}
It is analytically advantageous to express these relations in terms of Ramanujan class invariant $G_{n}$.
To simplify the resulting representations, the parameter $X_n$ is defined as
\begin{equation*}
X_n = \frac{G_{n}^{24}}{64(G_{n}^{24}-1)^2},
\end{equation*}
which allows us to express the series
\begin{equation}\label{eq:clausenramanujan4k}
\begin{aligned}
K^2(k_n) = \frac{\pi^2G_{n}^{12}}{4\sqrt{G_{n}^{24}-1}} \sum_{m=0}^{\infty} \frac{(-1)^m(4m)!}{m!^4} X_n^m.
\end{aligned}
\end{equation}
And its counterpart for the product of both integrals
\begin{equation}\label{eq:clausenramanujan4ek}
E(k_{n})K(k_{n}) = \frac{\pi^2}{8(G_{n}^{24}-1)} \sum_{m=0}^{\infty} \frac{(-1)^m (4m)!\left[ G_{n}^{12}\sqrt{G_{n}^{24}-1} + G_{n}^{24}(2m+1) + 2m \right]}{m!^4} X_n^m.
\end{equation}

Identity \eqref{eq:clausen4} arises from the theory of hypergeometric functions, enabling a connection to Ramanujan's alternative theory of signature 4, a topic extensively investigated by Berndt et al.~\cite{berndtBhargavaGarvan1995} 
and the Borweins~\cite{borwein1987}. By differentiating \eqref{eq:clausen4} with respect to $k$ and employing the derivative formula for $\frac{\mathrm{d}K}{\mathrm{d}k}$ given in \eqref{eq:derke}, we obtain the latter relation, which can subsequently be reformulated in terms of Ramanujan--Weber class invariants. It should be noted that this latter formula does not hold for $G_4$.
We employ Ramanujan--Weber class invariants as computed by Weber \cite{weber1895} and Ramanujan \cite{ramanujan1914}
\begin{equation}
\begin{aligned}
G_{5}  &= \left( \frac{1 + \sqrt{5}}{2} \right)^{1/4}, \quad & G_{7}  &= 2^{1/4}, \\
G_{13} &= \left( \frac{3 + \sqrt{13}}{2} \right)^{1/4}, \quad & G_{37} &= \left( 6 + \sqrt{37} \right)^{1/4}.
\end{aligned}
\end{equation}
From relation \eqref{eq:ramanujaninvariantrelation}, it follows that the corresponding singular moduli are given by
\begin{equation}
\begin{aligned}
k_{5}  &= \sqrt{\frac{1}{2}-\sqrt{\sqrt{5}-2}}, \quad & k_{7} &= \frac{\sqrt{2}}{8}(3-\sqrt{7}), \\
k_{13} &=\frac{1}{2}\sqrt{5\sqrt{13}-17}+\frac{\sqrt{26}-5\sqrt{2}}{4}, \quad &    k_{37} &= \frac{1}{2}\sqrt{145\sqrt{37}-881} + \frac{29\sqrt{2}-5\sqrt{74}}{4}.
\end{aligned}
\end{equation}
These values satisfy 
\begin{equation}\label{eq:ratios}
\frac{K(k'_{r})}{K(k_{r})}=\sqrt{r},
\end{equation}
for $r=5, 7, 13, 37$, respectively.
\section{Explicit Evaluations of Ramanujan Series in the Theory of Signature 4}
\begin{theorem}
We have the classical series representations for $\frac{1}{\pi}$
\begin{align}
\sum_{m=0}^{\infty}\frac{(-1)^{m}(4m)!}{2^{10m}m!^{4}}(20m+3) &= \frac{8}{\pi} \label{eq:ramanujank5}, \\[1.5ex]
\sum_{m=0}^{\infty}\frac{(-1)^m(4m)!}{63^{2m}m!^{4}}(65m+8)=\frac{9\sqrt{7}}{\pi}\label{eq:brendetalk7}, \\[1.5ex]
\sum_{m=0}^{\infty}\frac{(-1)^{m}(4m)!}{288^{2m}m!^{4}}(260m+23) &= \frac{72}{\pi} \label{eq:ramanujank13}, \\[1.5ex]
\sum_{m=0}^{\infty}\frac{(-1)^{m}(4m)!}{14112^{2m}m!^4}(21460m+1123) &= \frac{3528}{\pi}\label{eq:ramanujank37}.
\end{align}
\end{theorem}
\begin{proof}
\noindent\textbf{Proof of \eqref{eq:ramanujank5}.}
For this case we use that $k_5=\sqrt{\frac{1}{2}-\sqrt{\sqrt{5}-2}}$ and then the complementary modulus is $k'_{5}=\sqrt{\frac{1}{2}+\sqrt{\sqrt{5}-2}}$. 
Let us introduce the intermediate complex modular parameters $s_{5}$ and $l_{5}$ defined as
\begin{equation*}
s_{5}=\sqrt{4\sqrt{5}-8}-i(2-\sqrt{5}),
\end{equation*}
and
\begin{equation*}
l_{5}=i\left(\sqrt{5}+2-\sqrt{4\sqrt{5}+8}\right).
\end{equation*}
Thus, the corresponding complementary moduli are 
\begin{equation*}
s'_{5} = \frac{3\sqrt{2}-\sqrt{10}}{2} - i\sqrt{5\sqrt{5}-11},
\end{equation*}
and 
\begin{equation*}
l'_{5} = \sqrt{5\sqrt{5}+11} - \frac{3\sqrt{2}+\sqrt{10}}{2}.
\end{equation*}
Via Jacobi imaginary transformation \eqref{eq:jacobi_transke}, we have
\begin{equation}\label{eq:case5_1}
K(l_{5})=k'_{5}K(k_{5}).
\end{equation}
Using Jacobi reciprocal transformation \eqref{eq:jacobireciprocal} 
\begin{equation}\label{eq:case5_2}
K(l'_{5})=K\left(\frac{1}{k'_5}\right)\overset{\eqref{eq:jacobireciprocal}}=k'_5\left(K\left(k'_{5} \right) -iK\left(k_{5} \right) \right)
\overset{K(k'_{5})/K(k_{5})=\sqrt{5}}=(\sqrt{5}-i)k'_{5}K(k_{5}).
\end{equation}
And also we get
\begin{align}\label{eq:case5_3}
K(s'_{5})\overset{\eqref{eq:landeupwardke}}=\left(1-l_5\right)K\left(-l_{5}\right)=\left(1-l_5\right)K\left(l_{5}\right)\overset{\eqref{eq:case5_1}}=\left(k'_5-ik_{5}\right)K(k_5).
\end{align}
Using Jacobi reciprocal transformation and Landen descending transformation for the first-kind integral we have respectively 
\begin{equation}\label{eq:case5_r}
K(\frac{1}{s_5})\overset{\eqref{eq:jacobireciprocal}}=s_{5}\left(K(s_{5})-iK(s'_{5}) \right), \quad K(\frac{1}{s_5})\overset{\eqref{eq:landendownwardke}}=\frac{1}{1+\frac{1}{s_{5}}}K(\frac{1}{k'_5}). 
\end{equation}
Then combining Eqs.~\eqref{eq:case5_2}, \eqref{eq:case5_3} and \eqref{eq:case5_r}  we get
\begin{equation}\label{eq:case5_4}
K(s_{5})=\left[\left(1+\sqrt{5}-\sqrt{\sqrt{5}+2} \right)k'_{5}+ik_{5}\left(1+\sqrt{\sqrt{5}+2} \right) \right]K(k_{5}).
\end{equation}
Using Eqs.~\eqref{eq:case5_1}, \eqref{eq:case5_2}, \eqref{eq:case5_3}, and \eqref{eq:case5_4}, we deduce the following ratios
\begin{equation*}
\frac{K(s'_5)}{K(s_{5})}=\frac{\sqrt{5}}{3}-\frac{i}{3},\quad \frac{K(l'_{5})}{K(l_{5})}=\sqrt{5}-i.
\end{equation*}
Since $3\frac{K(s'_5)}{K(s_5)}=\frac{K(l'_5)}{K(l_5)}$, it follows that $l_5$ has degree 3 over $s_{5}$, which satisfies the conditions to apply \eqref{eq:deg3} with the nome $q=e^{-(\sqrt{5}/3-i/3)\pi}=i^{\frac{2}{3}}e^{-\frac{\sqrt{5}\pi}{3}}$. Substituting these into \eqref{eq:deg3}, we obtain
\begin{align}\label{eq:modular_k5_deg3}
3P(e^{-2\sqrt{5}\pi})-P\left(i^{\frac{4}{3}}e^{-\frac{2\sqrt{5}\pi}{3}}\right)=\frac{4K(s_5)K(l_5)}{\pi^2}(1+s_{5}l_{5}+s'_{5}l'_{5}) \nonumber \\
=\frac{K^2(k_{5})}{\pi^2}\left(\sqrt{136\sqrt{5}-248}+i\sqrt{104\sqrt{5}-232} \right),
\end{align}
where, in the final step, we have utilized the values of $s_{5}$,$s'_{5}$,$l_{5}$ and $l'_{5}$ alongside Eqs.~\eqref{eq:case5_1} and~\eqref{eq:case5_4}.
Subsequently, substituting $\tau = \frac{4}{9} - \frac{2\sqrt{5}i}{9}$ into \eqref{eq:reflection} yields
\begin{equation}\label{eq:modular_k5_1}
\left(\frac{4}{9} - \frac{2\sqrt{5}i}{9} \right)P\left(i^{\frac{4}{3}}e^{-\frac{2\sqrt{5}\pi}{3}}\right)+P(-e^{-\sqrt{5}\pi})=\frac{2\sqrt{5}}{\pi}-\frac{2i}{\pi}.
\end{equation}
Substituting $k = k_{5}$ into \eqref{eq:degree2_nome1} yields
\begin{equation}\label{eq:modular_k5_2}
2P(e^{-2\sqrt{5}\pi})-P(-e^{-\sqrt{5}\pi})=\frac{4K^2(k_{5})}{\pi^2}(1-2k^2_{5}).
\end{equation}
Similarly, substituting $k = k_{5}$ into \eqref{eq:degree2_nome2} yields
\begin{equation}\label{eq:modular_k5_3}
P(e^{-2\sqrt{5}\pi})=\frac{12E(k_5)K(k_5)}{\pi^2}+\frac{(4k_5^2-8)K^{2}(k_5)}{\pi^2}.
\end{equation}
Therefore, combining Eqs.~\eqref{eq:modular_k5_deg3}, \eqref{eq:modular_k5_1}, \eqref{eq:modular_k5_2}, and \eqref{eq:modular_k5_3}, we arrive at
\begin{samepage}
\begin{align*}
\frac{2\sqrt{5}}{\pi}-\frac{2i}{\pi} = \frac{40E(k_{5})K(k_{5})}{\pi^2}  &- \left(\sqrt{160\sqrt{5} - 160} + 20\right) \frac{K^2(k_{5})}{\pi^2} \\
&+ i \left[ \left(\sqrt{32\sqrt{5} - 32} + 4\sqrt{5}\right) \frac{K^2(k_{5})}{\pi^2} - \frac{8\sqrt{5}E(k_{5})K(k_{5})}{\pi^2} \right].
\end{align*}
\end{samepage}
Equating either the real or the imaginary parts on both sides of this final expression yields the same relation, which can be written by introducing $G_{5}$ as
\begin{equation}\label{eq:rel_n5}
 20E(k_{5})K(k_{5})-\left(\frac{4\sqrt{5}}{G_5^2}+10\right)K^2(k_{5})=\sqrt{5}\pi.
\end{equation}
Using Eqs.~\eqref{eq:clausenramanujan4k}, \eqref{eq:clausenramanujan4ek} and \eqref{eq:rel_n5}, then $X_{5}=\frac{1}{2^{10}}$ since $G_{5}=(\frac{1+\sqrt{5}}{2})^{1/4}$ and we get
\begin{align*}
\sqrt{5}\pi=20E(k_{5})K(k_{5})-\left(\frac{4\sqrt{5}}{G_5^2}+10\right)K^2(k_{5}) &= \frac{\sqrt{5}\pi^2}{8}\sum_{m=0}^{\infty}\frac{(-1)^{m}(4m)!}{2^{10m}m!^{4}}(20m+3) \nonumber \\
\implies \sum_{m=0}^{\infty}\frac{(-1)^{m}(4m)!}{2^{10m}m!^{4}}(20m+3) = \frac{8}{\pi}.
\end{align*}

\medskip
\noindent\textbf{Proof of \eqref{eq:brendetalk7}.}
For this case we use $k_{7}=\frac{\sqrt{2}}{8}(3-\sqrt{7})$ and its complementary modulus $k'_{7}=\frac{\sqrt{2}}{8}(3+\sqrt{7})$. Let us introduce the intermediate complex modular parameter
\begin{equation*}
l_{7}=\frac{3}{8}+\frac{\sqrt{7}i}{8},
\end{equation*}
then the complementary modulus is
\begin{equation*}
l'_{7}=\frac{3\sqrt{7}}{8}-\frac{i}{8}.
\end{equation*}
Using Landen ascending transformation and Jacobi imaginary transformation we obtain
\begin{equation}\label{eq:ratio1_k7}
K(l_{7})\overset{\eqref{eq:landeupwardke}}=(1+i(8-3\sqrt{7})K(i(8-3\sqrt{7}))\overset{\eqref{eq:jacobi_transke}}=(k'_{7}+ik_{7})K(k_{7}).
\end{equation}
We have also
\begin{align}\label{eq:ratio2_k7}
K(l'_{7})\overset{\eqref{eq:landeupwardke}}=(\frac{11}{8}-\frac{\sqrt{7}i}{8})K(\frac{3}{8}-\frac{\sqrt{7}i}{8})\overset{\eqref{eq:landeupwardke}}=(4-\sqrt{7}+i(4\sqrt{7}-11))K(i(3\sqrt{7}-8)) \nonumber \\
 \overset{\eqref{eq:jacobi_transke}}=\left[\frac{5\sqrt{2}+\sqrt{14}}{8}+i(\frac{\sqrt{14}-5\sqrt{2}}{8}) \right]K(k_{7}).
\end{align}
Using Eqs.~\eqref{eq:ratio1_k7} and \eqref{eq:ratio2_k7} we deduce the ratio
\begin{equation*}
\frac{K(l'_{7})}{K(l_{7})}=\frac{\sqrt{7}}{2}-\frac{i}{2}.
\end{equation*}
So the corresponding nome associated to $l_{7}$ is $q=e^{-\pi(\frac{\sqrt{7}}{2}-\frac{i}{2})}=ie^{-\sqrt{7}\pi/2}$. Then replacing this nome in Eq. \eqref{eq:degree2_nome3}
\begin{equation}\label{eq:modular_k7_1}
2P(-e^{-\sqrt{7}\pi})-P(ie^{-\sqrt{7}\pi/2})=\frac{4K^2(l_{7})}{\pi^2}(1+l_{7}^2)=\frac{K^2(k_{7})}{2\pi^2}\left(3\sqrt{7}+3i\right),
\end{equation}
where, in the final step, we have utilized the values of $l_{7}$ alongside Eq.~\eqref{eq:ratio1_k7}. Subsequently, substituting $\tau = \frac{3}{8}-\frac{\sqrt{7}i}{8}$ into \eqref{eq:reflection} yields
\begin{equation}\label{eq:modular_k7_2}
\left(\frac{3}{8}-\frac{\sqrt{7}i}{8}\right)P(ie^{-\sqrt{7}\pi/2})+P(-e^{-\sqrt{7}\pi})=\frac{3\sqrt{7}}{2\pi}-\frac{3i}{2\pi}.
\end{equation}
Substituting $k = k_{7}$ into \eqref{eq:degree2_nome1} yields
\begin{equation}\label{eq:modular_k7_3}
2P(e^{-2\sqrt{7}\pi})-P(-e^{-\sqrt{7}\pi})=\frac{4K^2({k_{7}})}{\pi^2}(1-2k_{7}^2).
\end{equation}
Similarly, substituting $k = k_{7}$ into \eqref{eq:degree2_nome2} yields
\begin{equation}\label{eq:modular_k7_4}
P(e^{-2\sqrt{7}\pi})=\frac{12E(k_{7})K(k_{7})}{\pi^2}+\frac{(4k_{7}^2-8)K^2(k_{7})}{\pi^2}.
\end{equation}
Therefore, combining Eqs.~\eqref{eq:modular_k7_1}, \eqref{eq:modular_k7_2}, \eqref{eq:modular_k7_3}, and \eqref{eq:modular_k7_4}, we arrive at
\begin{equation*}
\frac{3\sqrt{7}}{2\pi}-\frac{3i}{2\pi}=\frac{42E(k_{7})K(k_{7})}{\pi^2}-\frac{(6\sqrt{7}+21)K^2(k_{7})}{\pi^2}+i\left[\frac{(3\sqrt{7}+6)K^2(k_{7})}{\pi^2}-\frac{6\sqrt{7}K(k_{7})E(k_{7})}{\pi^2} \right].
\end{equation*}
Equating either the real or the imaginary parts on both sides of this final expression yields the same relation
\begin{equation}\label{eq:rel_n7}
28E(k_{7})K(k_{7})-(4\sqrt{7}+14)K^2(k_{7})=\sqrt{7}\pi.
\end{equation}
Using Eqs.~\eqref{eq:clausenramanujan4k}, \eqref{eq:clausenramanujan4ek} and \eqref{eq:rel_n7}, then $X_{7}=\frac{1}{63^2}$ since $G_{7}=2^{1/4}$ and we get
\begin{align*}
\sqrt{7}\pi = 28E(k_7)K(k_7) - (4\sqrt{7}+14)K^2(k_7) &= \frac{\pi^2}{9}\sum_{m=0}^{\infty}\frac{(-1)^m(4m)!}{63^{2m}m!^{4}}(65m+8) \nonumber \\
\implies \sum_{m=0}^{\infty}\frac{(-1)^m(4m)!}{63^{2m}m!^{4}}(65m+8)=\frac{9\sqrt{7}}{\pi}.
\end{align*}

\medskip
\noindent\textbf{Proof of \eqref{eq:ramanujank13}.}
For this case we use that $k_{13}=\frac{1}{2}\sqrt{5\sqrt{13}-17}+\frac{\sqrt{26}-5\sqrt{2}}{4}$ and then the complementary modulus is $k'_{13}=\frac{1}{2}\sqrt{5\sqrt{13}-17}+\frac{5\sqrt{2}-\sqrt{26}}{4}$. 
Let us introduce the intermediate complex modular parameters $s_{13}$ and $l_{13}$ defined as
\begin{equation*}
s_{13}=\sqrt{180\sqrt{13}-648}+i(5\sqrt{13}-18),
\end{equation*}
and
\begin{equation*}
l_{13}=i\left(18+5\sqrt{13}-\sqrt{180\sqrt{13}+648}\right).
\end{equation*}
Thus, the corresponding complementary moduli are
\begin{equation*}
s'_{13}=\frac{7\sqrt{26}-25\sqrt{2}}{2}-i\sqrt{185\sqrt{13}-667},
\end{equation*}
and 
\begin{equation*}
l'_{13}=\sqrt{185\sqrt{13}+667}-\frac{7\sqrt{26}+25\sqrt{2}}{2}.
\end{equation*}
As in the previous case, we require the complete elliptic integrals $K(s_{13})$, $K(s'_{13})$, $K(l_{13})$, and $K(l'_{13})$.

Via Jacobi imaginary transformation \eqref{eq:jacobi_transke} we obtain
\begin{equation}\label{eq:case13_1}
    K(l_{13})=k'_{13}K(k_{13}).
\end{equation}
And
\begin{equation}\label{eq:case13_2}
    K(l'_{13})=K(\frac{1}{k'_{13}})\overset{\eqref{eq:jacobireciprocal}}=k'_{13}\left(K\left(k'_{13} \right) -iK\left(k_{13} \right) \right)
\overset{K(k'_{13})/K(k_{13})=\sqrt{13}}=(\sqrt{13}-i)k'_{13}K(k_{13}).
\end{equation}
Furthermore, via \eqref{eq:landeupwardke}, we have
\begin{equation}\label{eq:case13_3}
   K(s'_{13})\overset{\eqref{eq:landeupwardke}}=\left(1-l_{13}\right)K\left(-l_{13}\right)=\left(1-l_{13}\right)K\left(l_{13}\right)\overset{\eqref{eq:case13_1}}=\left(k'_{13}-ik_{13}\right)K(k_{13}).
\end{equation}
Using Jacobi reciprocal transformation and Landen descending transformation for the first-kind integral we have respectively 
\begin{equation}\label{eq:case13_r}
K(\frac{1}{s_{13}})\overset{\eqref{eq:jacobireciprocal}}=s_{13}\left(K(s_{13})-iK(s'_{13}) \right), \quad K(\frac{1}{s_{13}})\overset{\eqref{eq:landendownwardke}}=\frac{1}{1+\frac{1}{s_{13}}}K(\frac{1}{k'_{13}}). 
\end{equation}
Then combining Eqs.~\eqref{eq:case13_2}, \eqref{eq:case13_3} and \eqref{eq:case13_r}  we get
\begin{equation}\label{eq:case13_4}
K(s_{13})=\left[\left(9+3\sqrt{13}-\sqrt{45\sqrt{13}+162} \right)k'_{13} +ik_{13}\left(9+2\sqrt{13}+\sqrt{45\sqrt{13}+162} \right)\right]K(k_{13}).
\end{equation}
Using Eqs.~\eqref{eq:case13_1}, \eqref{eq:case13_2}, \eqref{eq:case13_3}, and \eqref{eq:case13_4}, we deduce the following ratios
\begin{equation*}
    \frac{K(s'_{13})}{K(s_{13})}=\frac{\sqrt{13}}{7}-\frac{i}{7},\quad \frac{K(l'_{13})}{K(l_{13})}= \sqrt{13}-i.
\end{equation*}
Because these ratios fulfill the strict scaling requirement
\begin{equation*}
    7\frac{K(s'_{13})}{K(s_{13})}=\frac{K(l'_{13})}{K(l_{13})},
\end{equation*}
it follows that $l_{13}$ has degree 7 over $s_{13}$, which satisfies the conditions to apply \eqref{eq:deg7} with the nome $q=e^{-(\sqrt{13}/7-i/7)\pi}=(-1)^{1/7}e^{-\sqrt{13}\pi/7}$. To simplify the notation, we suppress the index $13$ within the brackets by setting $s \equiv s_{13}$ and $l \equiv l_{13}$. Substituting these into \eqref{eq:deg7} we obtain
\begin{align}
\label{eq:modular_k13_deg7}
7P \left( e^{-2\sqrt{13}\pi} \right) &- P \left( (-1)^{2/7} e^{-2\sqrt{13}\pi / 7} \right) = \frac{12 K(s) K(l)}{\pi^2} \left[ 1 + sl + s'l' \right] \nonumber \\
&= \frac{K^2(k_{13})}{\pi^2} \left[\sqrt{105768\sqrt{13}-380664}+i\sqrt{20232\sqrt{13} -72936}\right],
\end{align}
where, in the final step, we have utilized the values of $s_{13}$,$s'_{13}$,$l_{13}$ and $l'_{13}$ alongside Eqs.~\eqref{eq:case13_1} and~\eqref{eq:case13_4}.
Next, evaluating the functional reflection identity  \eqref{eq:reflection}  at the specific complex variable argument $\tau=\frac{12}{49}-\frac{2\sqrt{13}i}{49}$, we find
\begin{equation}\label{eq:modular_k13_1}
\left(\frac{12}{49}-\frac{2\sqrt{13}i}{49} \right)P\left( (-1)^{2/7} e^{-2\sqrt{13}\pi / 7}\right)+P(-e^{-\sqrt{13}\pi})=\frac{6\sqrt{13}}{7\pi}-\frac{6i}{7\pi}.
\end{equation}
Furthermore, specifying the singular value $k_{13}$ into the structural relations \eqref{eq:degree2_nome1} and \eqref{eq:degree2_nome2} yields the following two evaluations
\begin{equation}\label{eq:modular_k13_2}
    2P\left(e^{-2\sqrt{13}\pi}\right)-P\left(-e^{-\sqrt{13}\pi}\right)=\frac{4K^2({k_{13}})}{\pi^2}(1-2k^2_{13}),
\end{equation}
\begin{equation}\label{eq:modular_k13_3}
    P\left(e^{-2\sqrt{13}\pi}\right)=\frac{12E(k_{13})K(k_{13})}{\pi^2}+\frac{(4k_{13}^2-8)K^2(k_{13})}{\pi^2}.
\end{equation}
Therefore, combining Eqs.~\eqref{eq:modular_k13_deg7}, \eqref{eq:modular_k13_1}, \eqref{eq:modular_k13_2}, and \eqref{eq:modular_k13_3}, we arrive at
\begin{align*}
\frac{6\sqrt{13}}{7\pi}-\frac{6i}{7\pi} &= \frac{312E(k_{13})K(k_{13})}{7\pi^2} - \left(\sqrt{138528\sqrt{13} - 482976} + 156\right) \frac{K^2(k_{13})}{7\pi^2} \\
& \quad + i \left[ \left(\sqrt{10656\sqrt{13}-37152} +12\sqrt{13}\right) \frac{K^2(k_{13})}{7\pi^2} - \frac{24\sqrt{13}E(k_{13})K(k_{13})}{7\pi^2} \right].
\end{align*}

Equating either the real or the imaginary parts on both sides of this final expression yields the same relation, which can be written by introducing $G_{13}$ as
\begin{equation}\label{eq:rel_n13}
 52E(k_{13})K(k_{13})-\left(\frac{18\sqrt{13}-26}{G_{13}^2}+26\right)K^2(k_{13})=\sqrt{13}\pi.
\end{equation}
Using Eqs.~\eqref{eq:clausenramanujan4k}, \eqref{eq:clausenramanujan4ek} and \eqref{eq:rel_n13}, then $X_{13}=\frac{1}{288^{2}}$ since $G_{13}=(\frac{3+\sqrt{13}}{2})^{1/4}$ and we get
\begin{samepage}
\begin{align*}
\hspace{-1cm}\sqrt{13}\pi &= 52E(k_{13})K(k_{13})-\left(\frac{18\sqrt{13}-26}{G_{13}^2}+26\right)K^2(k_{13}) \\[1.5ex] 
&\hspace{-1cm}= \frac{\sqrt{13}\pi^2}{72}\sum_{m=0}^{\infty}\frac{(-1)^{m}(4m)!}{288^{2m}m!^{4}}(260m+23) \implies \sum_{m=0}^{\infty}\frac{(-1)^{m}(4m)!}{288^{2m}m!^{4}}(260m+23)=\frac{72}{\pi}.
\end{align*}
\end{samepage}

\medskip
\noindent\textbf{Proof of \eqref{eq:ramanujank37}.}
For this case we use
\begin{equation*}
    k_{37} = \frac{1}{2}\sqrt{145\sqrt{37}-881} + \frac{29\sqrt{2}-5\sqrt{74}}{4},\quad     k'_{37}=\frac{1}{2}\sqrt{145\sqrt{37}-881} + \frac{5\sqrt{74}-29\sqrt{2}}{4}.
\end{equation*}
Next, we introduce the intermediate complex modular parameters $s_{37}$ and $l_{37}$, defined as follows
\begin{equation*}
    s_{37}=42\sqrt{145\sqrt{37}-882}+i\left(145\sqrt{37}-882\right),
\end{equation*}
\begin{equation*}
   l_{37}=i\left(145\sqrt{37}+882-42\sqrt{145\sqrt{37}+882}\right).
\end{equation*}
Thus, the corresponding complementary moduli are
\begin{equation*}
s'_{37}=\frac{1247\sqrt{2}-205\sqrt{74}}{2}-i\sqrt{255925\sqrt{37}-1556731},
\end{equation*}
and 
\begin{equation*}
l'_{37}=\sqrt{255925\sqrt{37}+1556731}-\frac{1247\sqrt{2}+205\sqrt{74}}{2}.
\end{equation*}
As in the previous cases, we require the complete elliptic integrals $K(s_{37})$, $K(s'_{37})$, $K(l_{37})$, and $K(l'_{37})$.
Via Jacobi imaginary transformation \eqref{eq:jacobi_transke} we obtain
\begin{equation}\label{eq:case37_1}
    K(l_{37})=k'_{37}K(k_{37}).
\end{equation}
And
\begin{equation}\label{eq:case37_2}
    K(l'_{37})=K(\frac{1}{k'_{37}})\overset{\eqref{eq:jacobireciprocal}}=k'_{37}\left(K\left(k'_{37} \right) -iK\left(k_{37} \right) \right)
\overset{K(k'_{37})/K(k_{37})=\sqrt{37}}=(\sqrt{37}-i)k'_{37}K(k_{37}).
\end{equation}
Furthermore, via \eqref{eq:landeupwardke}, we have
\begin{equation}\label{eq:case37_3}
   K(s'_{37})\overset{\eqref{eq:landeupwardke}}=\left(1-l_{37}\right)K\left(-l_{37}\right)=\left(1-l_{37}\right)K\left(l_{37}\right)\overset{\eqref{eq:case37_1}}=\left(k'_{37}-ik_{37}\right)K(k_{37}).
\end{equation}
Using Jacobi reciprocal transformation and Landen descending transformation for the first-kind integral we have respectively 
\begin{equation}\label{eq:case37_r}
K(\frac{1}{s_{37}})\overset{\eqref{eq:jacobireciprocal}}=s_{37}\left(K(s_{37})-iK(s'_{37}) \right), \quad K(\frac{1}{s_{37}})\overset{\eqref{eq:landendownwardke}}=\frac{1}{1+\frac{1}{s_{37}}}K(\frac{1}{k'_{37}}). 
\end{equation}
Then combining Eqs.~\eqref{eq:case37_2}, \eqref{eq:case37_3} and \eqref{eq:case37_r}  we get
\begin{equation}\label{eq:case37_4}
\begin{split}
K(s_{37}) = \Bigg[ &\left(441+73\sqrt{37}-\sqrt{63945\sqrt{37}+388962} \right)k'_{37} \\
&+ik_{37}\left(441+72\sqrt{37}+\sqrt{63945\sqrt{37}+388962}\right)\Bigg]K(k_{37}).
\end{split}
\end{equation}
Using Eqs.~\eqref{eq:case37_1}, \eqref{eq:case37_2}, \eqref{eq:case37_3}, and \eqref{eq:case37_4}, we deduce the following ratios
\begin{equation*}
    \frac{K(s'_{37})}{K(s_{37})}=\frac{\sqrt{37}}{19}-\frac{i}{19},\quad \frac{K(l'_{37})}{K(l_{37})}= \sqrt{37}-i.
\end{equation*}
Because these ratios fulfill the strict scaling requirement
\begin{equation*}
    19\frac{K(s'_{37})}{K(s_{37})}=\frac{K(l'_{37})}{K(l_{37})},
\end{equation*}
it follows that $l_{37}$ has degree 19 over $s_{37}$, which satisfies the conditions to apply \eqref{eq:deg19} with the nome $q=e^{-(\sqrt{37}/19-i/19)\pi}=(-1)^{1/19}e^{-\sqrt{37}\pi/19}$. To simplify the notation, we suppress the index $37$ within the brackets by setting $s \equiv s_{37}$ and $l \equiv l_{37}$. Substituting these into \eqref{eq:deg19} we obtain
\begin{samepage}
\begin{align}\label{eq:modular_k37_deg19}
    19P \left( e^{-2\sqrt{37}\pi} \right) &- P \left((-1)^{2/19}e^{-2\sqrt{37}\pi / 19}  \right) \nonumber \\
    &=  \frac{24K(s)K(l)}{\pi^2}\left[1 + sl+ s'l' + \sqrt{sl} + \sqrt{s'l'} - \sqrt{ss'll'}\right]   \nonumber  \\
    &= \frac{6\sqrt{2}K^2(k_{37})}{\pi^2}\left[\sqrt{26744161\sqrt{37}-162678271}+i\sqrt{1494589\sqrt{37}-9091229} \right],
\end{align}
\end{samepage}
where, in the final step, we have utilized the values of $s_{37}$,$s'_{37}$,$l_{37}$ and $l'_{37}$ alongside Eqs.~\eqref{eq:case37_1} and~\eqref{eq:case37_4}.
Next, evaluating the functional reflection identity  \eqref{eq:reflection}  at the specific complex variable argument $\tau=\frac{36}{361}-\frac{2\sqrt{37}i}{361}$, we find
\begin{equation}\label{eq:modular_k37_1}
\left(\frac{36}{361}-\frac{2\sqrt{37}i}{361}\right)P\left((-1)^{2/19}e^{-2\sqrt{37}\pi/19}\right)+P\left(-e^{-\sqrt{37}\pi}\right)=\frac{6\sqrt{37}}{19\pi}-\frac{6i}{19\pi}.
\end{equation}
Furthermore, specifying the singular value $k_{37}$ into the structural relations \eqref{eq:degree2_nome1} and \eqref{eq:degree2_nome2} yields the following two evaluations
\begin{equation}\label{eq:modular_k37_2}
    2P\left(e^{-2\pi\sqrt{37}}\right)-P\left(-e^{-\sqrt{37}\pi}\right)=\frac{4K^2({k_{37}})}{\pi^2}(1-2k^2_{37}),
\end{equation}
\begin{equation}\label{eq:modular_k37_3}
    P\left(e^{-2\pi\sqrt{37}}\right)=\frac{12E(k_{37})K(k_{37})}{\pi^2}+\frac{(4k_{37}^2-8)K^2(k_{37})}{\pi^2}.
\end{equation}
Therefore, combining Eqs.~\eqref{eq:modular_k37_deg19}, \eqref{eq:modular_k37_1}, \eqref{eq:modular_k37_2}, and \eqref{eq:modular_k37_3}, we arrive at
\begin{equation*}
\begin{aligned}
    \frac{6\sqrt{37}}{19\pi}-\frac{6i}{19\pi} = &\frac{888E(k_{37})K(k_{37})}{19\pi^2}-\left(\frac{444}{19}+\sqrt{\frac{552332448\sqrt{37}-3359549088}{361}}\right)\frac{K^2(k_{37})}{\pi^2} \\
    &\hspace{-1cm}+i\left[\left(\frac{12\sqrt{37}}{19}+\sqrt{\frac{14927904\sqrt{37}-90798624}{361}} \right)\frac{K^2(k_{37})}{\pi^2} 
    -\frac{24\sqrt{37}E(k_{37})K(k_{37})}{19\pi^2}\right].
\end{aligned}
\end{equation*}
Equating either the real or the imaginary parts on both sides of this final expression yields the same relation, which can be written by introducing $G_{37}$ as
\begin{equation}\label{eq:rel_n37}
  148E(k_{37})K(k_{37})-\left(\frac{2\left(171\sqrt{37}-925\right)}{G_{37}^{2}}+74\right)K^2(k_{37})=\sqrt{37}\pi.
\end{equation}
Using Eqs.~\eqref{eq:clausenramanujan4k}, \eqref{eq:clausenramanujan4ek} and \eqref{eq:rel_n37}, then $X_{37}=\frac{1}{14112^{2}}$ since $G_{37}=(6+\sqrt{37})^{1/4}$ and we get
\begin{align*}
\sqrt{37}\pi &= 148E(k_{37})K(k_{37})-\left(\frac{2\left(171\sqrt{37}-925\right)}{G_{37}^{2}}+74\right)K^2(k_{37}) \\[1.5ex] 
&\hspace{-1cm}= \frac{\sqrt{37}\pi^2}{3528} \sum_{m=0}^{\infty}\frac{(-1)^{m}(4m)!}{14112^{2m}m!^4}(21460m+1123) \implies \sum_{m=0}^{\infty}\frac{(-1)^{m}(4m)!}{14112^{2m}m!^4}(21460m+1123) = \frac{3528}{\pi}.
\end{align*}

\end{proof}
The series \eqref{eq:brendetalk7} was established by Berndt, Chan, and Liaw \cite{berndtChanLiaw2001}. The remaining identities, \eqref{eq:ramanujank5}, \eqref{eq:ramanujank13}, and \eqref{eq:ramanujank37}, appear as formulas (35), (37), and (39), respectively, in Ramanujan's work \cite{ramanujan1914}.

\section{Conclusion}
The analytical approach presented in this work establishes an alternative path for deriving several Ramanujan-type series for $1/\pi$ via manageable lower-degree reductions facilitated by the introduction of complex numbers. By deconstructing the evaluation of the cases $r = 5, 7, 13,$ and $37$, we demonstrate that complex algebraic structures can be entirely resolved without appealing to higher-degree modular transformations obviating the need for degrees $5, 7, 13,$ and $37$ typically employed in standard proofs of these series. 

In particular, we explicitly explain the origin of the coefficients $21460$ and $1123$ in the Ramanujan series \eqref{eq:ramanujank37} without relying on computational assistance. This strongly suggests that Ramanujan's approach was systematic and that he deduced these coefficients through purely analytical, rather than empirical, means.

\backmatter

\section*{Declarations}
\begin{itemize}
    \item \textbf{Acknowledgments:} The author wishes to express his sincere gratitude to Paramanand Singh for his highly valuable and insightful comments. Generative AI was used solely for English language translation and linguistic polishing of the manuscript.
    \item \textbf{Funding:} The author received no financial support for the research, authorship, and/or publication of this article.
    \item \textbf{Conflict of interest:} The author declares no competing interests.
    \item \textbf{Ethical approval:} Not applicable.
    \item \textbf{Data availability:} Data sharing is not applicable to this article as no datasets were generated or analyzed during the current study.
    \item \textbf{Authors' contributions:} Entirely prepared by the sole author.
\end{itemize}

\end{document}